\documentclass[letterpaper]{gtpart}

\usepackage{pinlabel}
\usepackage[all,cmtip]{xy}

\title{The period-index problem for twisted topological $K$-theory}

\author{Benjamin Antieau}
\givenname{Benjamin}
\surname{Antieau}
\address{University of Washington\\Department of Mathematics\\Box 354350\\Seattle, WA 98195}
\email{benjamin.antieau@gmail.com}

\author{Ben Williams}
\givenname{Ben}
\surname{Williams}
\address{University of British Columbia\\Department of Mathematics\\1984 Mathematics
Road\\Vancouver, B.C., Canada V6T 1Z2}
\email{tbjw@math.ubc.edu}


\usepackage{amsmath}
\usepackage{amscd}
\usepackage{amsbsy}
\usepackage{amssymb}
\usepackage{verbatim}
\usepackage{eufrak}
\usepackage{microtype}
\usepackage{hyperref}
\usepackage{mathrsfs}

\usepackage{color}
\definecolor{todo}{rgb}{1,0,0}
\definecolor{conditional}{rgb}{0,1,0}
\definecolor{e-mail}{rgb}{0,.40,.80}
\definecolor{reference}{rgb}{.20,.60,.22}
\definecolor{mrnumber}{rgb}{.80,.40,0}
\definecolor{citation}{rgb}{0,.40,.80}

\DeclareMathAlphabet{\mathpzc}{OT1}{pzc}{m}{it}
\usepackage{dsfont}

\DeclareMathOperator{\Ext}{Ext}

\DeclareMathOperator{\End}{End}

\DeclareMathOperator{\PGL}{PGL}
\DeclareMathOperator{\Spin}{Spin}
\DeclareMathOperator{\SU}{SU}

\DeclareMathOperator{\GL}{GL}

\DeclareMathOperator{\tors}{tors}

\DeclareMathOperator{\ord}{ord}

\DeclareMathOperator{\Spec}{Spec}

\DeclareMathOperator{\Pic}{Pic}

\DeclareMathOperator{\NS}{NS}
\DeclareMathOperator{\Hoh}{H}
\DeclareMathOperator{\Eoh}{E}

\DeclareMathOperator{\Vect}{Vect}
\DeclareMathOperator{\Hom}{Hom}
\DeclareMathOperator{\Br}{Br}

\DeclareMathOperator*{\holim}{holim}
\DeclareMathOperator*{\colim}{colim}
\DeclareMathOperator*{\hocolim}{hocolim}
\DeclareMathOperator{\Sing}{Sing}

\DeclareMathOperator{\eti}{eti}
\DeclareMathOperator{\ind}{ind}
\DeclareMathOperator{\degr}{deg}
\DeclareMathOperator{\per}{per}
\DeclareMathOperator{\sk}{sk}

\newcommand{\topo}{\mathrm{top}}

\newcommand{\an}{\mathrm{an}}
\newcommand{\et}{\mathrm{\acute{e}t}}

\DeclareFontEncoding{OT2}{}{}

\newcommand{\rHoh}{\tilde{\Hoh}}

\newcommand{\msphere}{\mathrm{M}}
\newcommand{\rsphere}{\mathrm{R}}
\newcommand{\tsphere}{\mathrm{T}}
\newcommand{\sphere}{\mathrm{S}}

\newcommand{\Map}{\mathbf{Map}}

\newcommand{\Sp}{\mathrm{Sp}}
\newcommand{\K}{\mathbf{K}}

\DeclareMathOperator{\KU}{KU}
\DeclareMathOperator{\ku}{ku}

\newcommand{\Mod}{\mathrm{Mod}}
\newcommand{\Line}{\mathrm{Line}}

\newcommand{\rwe}{\overset\sim\rightarrow}

\newcommand{\riso}{\overset\simeq\rightarrow}

\newcommand{\we}{\simeq}
\newcommand{\iso}{\cong}
\newcommand{\isomto}{\riso}

\newcommand{\Shv}{\mathbf{Shv}}

\newcommand{\OX}{\mathcal{O}_X}

\newcommand{\CC}{\mathds{C}}

\newcommand{\QQ}{\mathds{Q}}
\newcommand{\ZZ}{\mathds{Z}}
\newcommand{\NN}{\mathds{N}}
\newcommand{\FF}{\mathds{F}}
\newcommand{\PP}{\mathds{P}}

\newcommand{\Gm}{\mathds{G}_{m}}

\renewcommand{\Z}{\mathds{Z}}

\usepackage{amsthm}

\theoremstyle{plain}
\newtheorem{theorem}{Theorem}[section]
\newtheorem{lemma}[theorem]{Lemma}
\newtheorem{proposition}[theorem]{Proposition}

\newtheorem{conjecture}[theorem]{Conjecture}
\newtheorem{corollary}[theorem]{Corollary}

\newtheorem{problem}[theorem]{Problem}

\theoremstyle{definition}
\newtheorem{definition}[theorem]{Definition}

\newtheorem{example}[theorem]{Example}

\theoremstyle{remark}
\newtheorem{remark}[theorem]{Remark}

\usepackage{tikz}
\usetikzlibrary{matrix,arrows}

\newcommand{\IEoh}{{}_{I}\!\Eoh}
\newcommand{\IIEoh}{{}_{II}\!\Eoh}

\let\oldmarginpar\marginpar
\renewcommand\marginpar[1]{\-\oldmarginpar[\raggedleft\footnotesize #1]%
{\raggedright\footnotesize #1}}

\begin{document}
\begin{abstract}
  \noindent
We introduce and solve a period-index problem for the Brauer group of a topological space.
The period-index problem is to relate the order of a class in the Brauer group to the degrees of Azumaya algebras representing it.
For any space of dimension $d$, we give upper bounds on the index depending only on $d$ and the order of the class.
By the Oka principle, this also solves the period-index problem for the analytic Brauer group of any Stein
space that has the homotopy type of a finite CW-complex.
Our methods use twisted topological $K$-theory, which was first introduced by Donovan and Karoubi.
We also study the cohomology of the projective unitary groups to give cohomological obstructions to a class being represented
by an Azumaya algebra of degree $n$.
Applying this to the finite skeleta of the Eilenberg-MacLane space $K(\ZZ/\ell,2)$, where $\ell$ is a prime, we construct
a sequence of spaces with an order $\ell$ class in the Brauer group, but whose indices tend to infinity.

\end{abstract}

\maketitle

\section{Introduction}

This paper gives a solution to a period-index problem for twisted topological $K$-theory. The solution should be viewed
as an existence theorem for twisted vector bundles.

Let $X$ be a connected CW-complex.  An Azumaya algebra $\mathcal{A}$ of degree $n$ on $X$ is a non-commutative algebra over the sheaf $\mathcal{C}$ of
complex-valued functions on $X$ such that $\mathcal{A}$ is a vector bundle of rank $n^2$ and the stalks are finite dimensional complex
matrix algebras $M_n(\CC)$. Examples of Azumaya algebras include the sheaves of
endomorphisms of complex vector bundles and the complex Clifford bundles $\CC l(E)$ of oriented even-dimensional real vector bundles $E$.  The Brauer
group $\Br(X)$ classifies topological Azumaya algebras on $X$ up to the usual Brauer equivalence: $\mathcal{A}_0$ and $\mathcal{A}_1$ are
Brauer equivalent if there exist vector bundles $\mathcal{E}_0$ and $\mathcal{E}_1$ and an isomorphism
\begin{equation*}
    \mathcal{A}_0\otimes_{\mathcal{C}}\End(\mathcal{E}_0)\iso\mathcal{A}_1\otimes_{\mathcal{C}}\End(\mathcal{E}_1)
\end{equation*}
of sheaves of $\mathcal{C}$-algebras.
Define $\Br(X)$ to be the free abelian group on isomorphism classes of Azumaya algebras modulo Brauer equivalence.

The group $\Br(X)$ is a subgroup of the cohomological Brauer group $\Br'(X)=\Hoh^3(X,\ZZ)_{\tors}$. This inclusion is obtained in the
following way. There is an exact sequence
\begin{equation*}
    1\rightarrow \CC^*\rightarrow\GL_n\rightarrow\PGL_n\rightarrow 1
\end{equation*}
which induces an exact sequence in non-abelian cohomology
\begin{equation*}
    \Hoh^1(X,\GL_n)\rightarrow\Hoh^1(X,\PGL_n)\rightarrow\Hoh^2(X,B\CC^*)=\Hoh^3(X,\ZZ).
\end{equation*}
The pointed set $\Hoh^1(X,\PGL_n)$ classifies degree-$n$ Azumaya algebras on $X$ up to isomorphism whereas the left arrow sends an
$n$-dimensional complex vector bundle $E$ to $\End(E)$.  The map from the free abelian group on isomorphism classes of Azumaya algebras to $\Hoh^3(X,\ZZ)$
thus factors through the Brauer group.

By a result of Serre~\cite{grothendieck_brauer_1}, every cohomological Brauer class $\alpha\in\Br'(X)=\Hoh^3(X,\ZZ)_{\tors}$ is represented
by topological Azumaya algebras of varying degrees when $X$ is a finite CW-complex, so that $\Br(X)=\Br'(X)$. The period-index problem is to determine which degrees arise for
a given class, $\alpha$.  The index $\ind(\alpha)$ is defined to be the greatest common divisor of these degrees. The period $\per(\alpha)$ is the order of
$\alpha$ in $\Br'(X)$. For any $\alpha$, one has $\per(\alpha)|\ind(\alpha)$.

As an example, consider
the Clifford bundle $\CC l(E)$ associated to an oriented $2n$-dimensional real vector bundle, $E$. This has class in the cohomological Brauer group
given by $W_3(E)$, the third integral Stiefel-Whitney class of $E$, and so $W_3(E)$ is either of period $1$ or $2$, depending on whether or not $E$
supports a $\Spin^c$-structure. The rank of $\CC l(E)$ is $2^{2n}$. If $\per(W_3(E))=2$, we find
\begin{equation*}
    2|\per(W_3(E))|\ind(W_3(E))|2^n.
\end{equation*}

To our knowledge, this topological period-index problem has not been considered before, although the parallel question in algebraic geometry has been the subject of a great deal
of work. For instance,
see~\cite{artin_brauer_1982, colliot_exposant_2002,colliot_brauersche_2001,de_jong_period_2004,becher_symbol_2004,lieblich_period_2007,lieblich_twisted_2008}.

By analyzing the cohomology of the universal period-$r$ cohomological Brauer class, $K(\ZZ/r,2)\xrightarrow{\beta}K(\ZZ,3)$,
we show that the period and index have the same prime divisors, as in the algebraic case of fields.
Then, using ideas from Antieau~\cite{antieau_index_2009},
we establish upper bounds on $\ind(\alpha)$ depending only on $\per(\alpha)$ and the dimension of $X$.

By studying the cohomology of the
projective unitary groups, we obtain obstructions to the representation of a class $\alpha \in \Br(X)$ by an Azumaya algebra of degree $n$.
Using these obstructions, we then construct families of examples of period-$r$ Brauer class whose indices form an unbounded sequence.

Suppose that $X$ is a $d$-dimensional finite CW-complex, meaning that there are no cells in
$X$ of dimension bigger than $d$.
Let $\alpha\in\Hoh^3(X,\ZZ)_{\tors}$, and let $r=\per(\alpha)$. The classifying space
$B\ZZ/r$ can be constructed as a countable CW complex with finitely many cells in each
dimension. It follows that its stable homotopy groups $\pi_j^s(B\ZZ/r)$ are all finitely generated abelian
groups. But, they are also torsion groups for $j>0$. Let $e_j^\alpha$ denote the exponent of the finite abelian group $\pi_j^s(B\ZZ/r)$.

This is our main theorem:

\begin{theorem}[theorem~\ref{thm:bound}] $\ind(\alpha)|\prod_{j=1}^{d-1}e_j^{\alpha}$.
\end{theorem}

The theorem has the following corollary, which relies on the computation in~\cite{antieau_index_2009} of the stable homotopy groups $\pi_j^s B\ZZ/\ell^k$ in a range.

\begin{theorem}[theorem~\ref{thm:toperind}]
    Let $X$ be a $d$-dimensional finite CW-complex, let $\ell$ be a prime such that $2\ell>d+1$, and suppose $\alpha\in\Hoh^3(X,\ZZ)_{\tors}$
    satisfies $\per(\alpha)=\ell^k$ for some $k$. Then,
    \begin{equation*}
        \ind(\alpha)|\per(\alpha)^{[\frac{d}{2}]}.
    \end{equation*}
\end{theorem}

Since the period and index are homotopy invariant, these theorems and all other statements
in the paper hold for any $X$ homotopy equivalent to a $d$-dimensional finite CW complex.

By the Oka principle~\cite{grauert_remmert}, which says that the topological and analytic classification of torsors of a complex Lie group agree on a Stein space,
the same theorems hold for the analytic period-index problem on a Stein space having the homotopy type of a finite CW-complex.

To prove these theorems, we study the Atiyah-Hirzebruch spectral sequence for twisted $K$-theory $\KU(X)_{\alpha}$
\begin{equation*}
    \Eoh_2^{p,q}=\Hoh^p(X,\ZZ(q/2))\Rightarrow \KU^{p+q}(X)_{\alpha}.
\end{equation*}
Here, twisted topological $K$-theory refers to the theory first introduced by Donovan-Karoubi~\cite{donovan_karoubi} and subsequently developed by Rosenberg~\cite{rosenberg_continuous},
Atiyah and Segal~\cite{atiyah-segal}, and others.
When $X$ is compact, the index of $\alpha$ is also the generator of the image of the edge map
\begin{equation*}
    \KU^0(X)_{\alpha}\rightarrow\Hoh^0(X,\ZZ).
\end{equation*}
We look for permanent cycles in $\Hoh^0(X,\ZZ)$; equivalently, we study the differentials leaving this group.

The study of such differentials is intricate. We should like to have a twisted analogue of the unit map $\pi_i^s\rightarrow \KU^{-i}$ but
unfortunately the unit map cannot be twisted when $\alpha$ is non-trivial. Instead, if $\alpha$ is $r$-torsion, there is a spectrum,
$\sphere[\ZZ/r]$, resembling a finite cover of the sphere spectrum, and a map $\sphere[\ZZ/r]\rightarrow \KU$ that extends the unit map and
which may be twisted. We obtain from this a map $\sphere[\ZZ/r](X)_{\sigma}\rightarrow \KU(X)_{\alpha}$, where $\sigma$ is a lift of $\alpha$
to $\Hoh^2(X,\ZZ/r)$.

The utility of this map is that the homotopy groups of $\sphere[\ZZ/r]$ are all finite torsion groups and we know the torsion in low
degrees (relative to $r$) by~\cite{antieau_index_2009}, which is also where this cycle of ideas originates.
In the Atiyah-Hirzebruch spectral sequence for $\sphere[\ZZ/r]_{\sigma}$, we obtain bounds on the differentials departing $\Hoh^0(X,\ZZ)=\Hoh^0(X,\sphere[\ZZ/r]^0)$.
The result follows by considering a natural morphism of Atiyah-Hirzebruch spectral sequences.

To give lower bounds on the index, we consider when a map $\alpha: X\rightarrow K(\ZZ,3)$ factors through $BPU_n\rightarrow K(\ZZ,3)$, which
is to say, when a class $\alpha \in \Hoh^3(X, \ZZ)$ may be represented by a degree-$n$ Azumaya algebra. If such a factorization exists, then we obtain
\begin{equation*}
    \Omega X\rightarrow PU_n\xrightarrow{\sigma_n} K(\ZZ,2).
\end{equation*}
The class $\sigma_n$ in $\Hoh^2(PU_n,\ZZ)$ is studied in~\cite{baum-browder}. There, the order of $\sigma_n^i$ is determined for all $i$. In particular, $\sigma_n^{n-1}$ is non-zero,
but $\sigma_n^n=0$. This leads to a necessary condition for the factorization of $X\rightarrow K(\ZZ,3)$ through $BPU_n$.

In another direction, the classifying space for $r$-torsion elements of $\Hoh^3(X,\ZZ)_{\tors}$ is the Eilenberg-MacLane space $K(\ZZ/r,2)$
together with the Bockstein
\begin{equation*}
    \beta\in\Hoh^3(K(\ZZ/r,2),\ZZ),
\end{equation*}
which is of order $r$. In order to use the Atiyah-Hirzebruch spectral sequence to study the period--index problem, we must understand the differentials
in the $\beta$-twisted Atiyah-Hirzebruch spectral sequence for $K(\ZZ/r,2)$. In particular, we want to study the differentials
\begin{equation*}
    d_{2k+1}^{\beta}:\Hoh^0(K(\ZZ/r,2),\ZZ)\dashrightarrow\Hoh^{2k+1}(K(\ZZ/r,2),\ZZ).
\end{equation*}
It is known by Atiyah and Segal~\cite{atiyah_twisted_2006} that $d_{3}^{\beta}(1)=\pm \beta$.

Atiyah and Segal also investigate the higher differentials for the twist of $K$-theory on
$K(\QQ,3)$, but the higher differentials for torsion twists have not been studied before.
In the torsion case, we show that there are infinitely many non-zero differentials leaving 
\[\Eoh_k^{0,0}\subseteq\Hoh^0(K(\ZZ/r,2),\ZZ),\]
which are all themselves necessarily torsion.

In fact, for a fixed prime $l$, we show that, after $d_3^{\beta}(1)=l$, one of the following $d_5^{\beta}(l),\ldots,d_{2l+1}^{\beta}(l)$ is non-zero in the $\beta$-twisted
spectral sequence for $K(\ZZ/l,2)$. Combining this with our results from the cohomology of the projective unitary groups, we find that $d_5^{\beta}(2)$ is non-zero in the $\beta$-twisted Atiyah-Hirzebruch
spectral sequence for $K(\ZZ/2,2)$.

In~\cite{antieau_index_2009}, the first named author used similar methods to define an \'etale index $\eti(\alpha)$ for classes of $\Br(X_{\et})$ and to give upper bounds in terms of the period.
In that work, the question of whether the \'etale index ever differs from the period was left unanswered. The examples of section~\ref{sec:4.2} and the comparison map of theorem~\ref{thm:comparison}
give evidence that $\eti\neq\per$ as invariants on $\Br(X_{\et})$.

Two papers that rely on this paper and extend its results have already appeared. The first is~\cite{aw2}, which shows that indeed the
\'etale index \emph{does} differ from the period. The second is~\cite{aw3}, which completely solves the topological period-index problem for
finite $6$-dimensional CW complexes.

The paper is organized as follows. In the rest of this section, we give background on the Brauer group and the period-index problem in various settings.
In section~\ref{sec:twistk} we establish the relevant technical tools in twisted $K$-theory we need, a comparison map
from twisted \'etale $K$-theory to twisted topological $K$-theory, as well as the twisted unit map. In section~\ref{sec:perind},
we prove the main general upper bounds. In section~\ref{sec:cohomology}, after an excursion into elementary number theory, we study the cohomology of the projective unitary spaces, and we apply this
to furnish spaces with a class of period $r$ and with arbitrarily large index.

\paragraph{Acknowledgments}
We thank Aravind Asok and Christian Haesemeyer for discussions and useful suggestions.

\subsection{The Brauer group}
For generalities on Azumaya algebras and Brauer groups on locally ringed sites, see Grothendieck~\cite{grothendieck_brauer_1}.
For example, an Azumaya algebra on a topological space is a sheaf of algebras over the continuous complex-valued functions that is
locally isomorphic to a matrix algebra over $\CC$.

Throughout, $X$ will denote a connected scheme, a complex analytic space, or a topological space. In any of
these cases, there is a Brauer group $\Br(X)$ consisting of Brauer-equivalence classes of
algebraic, analytic, or topological Azumaya algebras on $X$.

In each case, there is also a cohomological Brauer group, $\Br'(X)$, defined as $\Hoh^2(X_{\et},\Gm)_{\tors}$
when $X$ is algebraic, $\Hoh^2(X,\OX^{\ast})_{\tors}$ when $X$ is a complex analytic space, and
\begin{equation*}
    \Hoh^2(X,\mathcal{C}_X^\ast)_{\tors}=\Hoh^2(X,\CC^*)_{\tors}=\Hoh^3(X,\ZZ)_{\tors}
\end{equation*}
when $X$ is a topological space. There are natural inclusions
\begin{equation*}
    \Br(X)\subseteq\Br'(X)
\end{equation*}
in each case.

Given a complex algebraic scheme $X$, write $\Br(X_{\et})$ for the algebraic Brauer group of $X$, $\Br(X_{\an})$ for the analytic Brauer group of $X$, and $\Br(X_{\topo})$
for the topological Brauer group, and similarly for $\Br'$. There is a natural commutative diagram
\begin{equation*}
    \begin{CD}
        \Br(X_{\et})  @>>>    \Br(X_{\an})    @>>>    \Br(X_{\topo})\\
        @VVV            @VVV                    @VVV\\
        \Br'(X_{\et}) @>>>    \Br'(X_{\an})   @>>>    \Br'(X_{\topo}).
    \end{CD}
\end{equation*}
Similarly, if $X$ is a complex analytic space, let $\Br(X_{\an})$ and $\Br(X_{\topo})$ be the analytic and topological Brauer groups, respectively. There is a natural commutative diagram
\begin{equation*}
    \begin{CD}
        \Br(X_{\an})    @>>>    \Br(X_{\topo})\\
        @VVV                    @VVV\\
        \Br'(X_{\an})   @>>>    \Br'(X_{\topo}).
    \end{CD}
\end{equation*}

If $X$ is a complex algebraic scheme, and $\alpha\in\Br'(X_{\et})$, let $\alpha_{\an}$ and $\alpha_{\topo}$ be the corresponding elements in the analytic and topological cohomological Brauer groups.

\begin{example}
    If $X$ is a complex K3 surface, then $\Br(X_{\et})\iso(\QQ/\ZZ)^{22-\rho}$, where $0\leq \rho\leq 20$ is the rank of the Neron-Severi group $\NS(X)$~\cite[Section 3]{grothendieck_brauer_2}.
    But $\Br(X_{\topo})=0$ because $\Hoh^3(X,\ZZ)=0$.
\end{example}

\begin{example}
  If $X$ is a smooth projective rationally connected complex threefold, then $\Br(X_{\et})=\Br(X_{\an})=\Br(X_{\topo})$, a finite abelian
  group, because $\Hoh^2(X,\mathcal{O}_X)=\Hoh^3(X,\mathcal{O}_X)=0$.  Artin--Mumford~\cite{artin-mumford} used the Brauer group to give
  some of the first examples of smooth projective unirational threefolds which are not rational.  They found a threefold with non-trivial
  $2$-torsion in $\Br(X_{\topo})$, whereas the topological Brauer group is a birational invariant of smooth projective complex
  varieties~\cite[Proposition 1]{artin-mumford}, and $\Br(\PP^n_{\topo})=0$.
\end{example}

\subsection{A problem of Grothendieck}

\begin{problem}[Grothendieck]
    Determine when $\Br(X)\rightarrow\Br'(X)$ is surjective.
\end{problem}

Here is a summary of known results:
\begin{itemize}
    \item If $X$ is a scheme having an ample family of line bundles, then Gabber, in unpublished work, and de~Jong~\cite{de_jong_result_2003} have each shown that $\Br(X_{\et})=\Br'(X_{\et})$.
        This is the case, for instance, for quasi-projective schemes over affine schemes.
    \item The map is not necessarily surjective for non-separated schemes, by an example of~\cite{edidin_brauer_2001}.
        The example $Q$ is an affine quadric cone glued to itself along the smooth locus. In this case, $\Br(Q_{\et})=0$, while $\Br'(Q_{\et})=\ZZ/2$.
    \item The best result for complex spaces is the result of Schr\"oer~\cite{schroer_topological_2005} which gives a purely topological condition that ensures $\Br(X_{\an})=\Br'(X_{\an})$.
        This condition applies to complex Lie groups, Hopf manifolds, and all compact complex surfaces except for a class which conjecturally does not exist,
        namely the class VII surfaces that are not Hopf surfaces, Inoue surfaces, or surfaces containing a global spherical shell. Some special cases of Schr\"oer's theorem
        had been obtained by Iversen~\cite{iversen_brauer_1976}, Hoobler~\cite{hoobler_brauer_1972}, Berkovi\v{c}~\cite{berkovic_brauer_1972}, Elencwajg-Narasimhan~\cite{elencwajg-narasimhan},
        and Huybrechts-Schr\"oer~\cite{huybrechts-schroer}.
    \item   Serre~\cite[Th\'eor\`eme~1.6]{grothendieck_brauer_1} showed that if $X$ has the homotopy type of a finite CW-complex, then
        \begin{equation*}
            \Br(X_{\topo})=\Br'(X_{\topo})=\Hoh^3(X,\ZZ)_{\tors}.
        \end{equation*}
        In particular, this holds for open sets in compact manifolds.
    \item By the Oka principle, the analytic and topological classification of $PU_n$-torsors is the same over a Stein space. Thus, if $X$ is a (separated) Stein space having the homotopy type
        of a finite CW-complex, the result of Serre shows that
        \begin{equation*}
            \Br(X_{\an})=\Br'(X_{\an}).
        \end{equation*}
        In particular, this holds for Stein submanifolds of compact complex manifolds.
    \item   As a negative example, if $X=K(\ZZ/m,2)$, then $\Br(X_\topo)=0$, while
       \[\Br'(X_{\topo})=\ZZ/m.\] See corollary~\ref{c:kxz}.
\end{itemize}

\subsection{The prime divisors problem}

\begin{definition}
    For any class $\alpha\in\Br'(X)$, the period of $\alpha$, denoted $\per(\alpha)$, is the order of $\alpha$ in the group $\Br'(X)$.
\end{definition}

In general, the degree of an Azumaya algebra $\mathcal{A}$ is the positive square-root of its rank as an $\OX$-module. This is a locally constant
integer, and hence a class in $\Hoh^0(X,\ZZ)$. Henceforth, for simplicity, assume that $X$ is connected, so that $\degr(\mathcal{A})\in\ZZ$. It is
a general fact that if the class of $\mathcal{A}$ is $\alpha\in\Br(X)$, then
\begin{equation*}
    \per(\alpha)|\degr(\mathcal{A}).
\end{equation*}
Indeed, $A$ determines a $\PGL_n$-torsor, where $n=\degr(\mathcal{A})$. The class $\alpha\in\Br(X)\subseteq\Br'(X)$ is the coboundary of this torsor
\begin{equation*}
    \Hoh^1(X,\PGL_n)\rightarrow\Hoh^2(X,\Gm),
\end{equation*}
and this factors through
\begin{equation*}
    \Hoh^2(X,\mu_n)\rightarrow\Hoh^2(X,\Gm),
\end{equation*}
assuming that $n$ is prime to the characteristics of the residue fields of $X$. For details, see~\cite{grothendieck_brauer_1}.

\begin{definition}
    If $\alpha\in\Br(X)$, define the index of $\alpha$ to be
    \begin{equation*}
        \ind(\alpha)=\gcd\{\degr(\mathcal{A}):\mathcal{A}\in\alpha\}.
    \end{equation*}
    If $\alpha\in\Br'(X)\backslash\Br(X)$, set $\ind(\alpha)=+\infty$.
\end{definition}

In general, $\per(\alpha)|\ind(\alpha)$.

\begin{problem}
  Do $\per(\alpha)$ and $\ind(\alpha)$ have the same prime divisors?
\end{problem}

Known results:
\begin{itemize}
    \item   If $k$ is a field, $\alpha\in\Br(k)=\Br'(k)$, then the period and index of $\alpha$ have the same prime divisors.
    \item   If $X=\Spec A$ is an affine scheme, then Gabber has shown that $\per(\alpha)$ and $\ind(\alpha)$ have the same prime divisors.
    \item   An unpublished argument of Saltman using extension of coherent sheaves shows that the period and index have the same prime divisors for Brauer classes on regular noetherian
        irreducible schemes.
\end{itemize}

We prove below, in corollary \ref{c:indPerPrimesAgree}, that if $X$ is a finite CW-complex, and if $\alpha \in \Br(X)$,
then $\per(\alpha)$ and $\ind(\alpha)$ have the same prime divisors.

\subsection{The period-index problem}

A well-known problem for fields, where $\per(\alpha)$ and $\ind(\alpha)$ have the same prime divisors, is to bound the smallest
integers $e(\alpha)$ such that $\ind(\alpha)|\per(\alpha)^{e(\alpha)}$.

\begin{conjecture}[Colliot-Th\'el\`ene] Let $k$ be either a $C_d$-field or the function field of a $d$-dimensional variety over
    an algebraically closed field. Let $\alpha\in\Br(k)$, and suppose that $\per(\alpha)$ is prime to the characteristic
    of $k$. Then,
    \begin{equation*}
        \ind(\alpha)|\per(\alpha)^{d-1}.
    \end{equation*}
\end{conjecture}

Results:
\begin{itemize}
    \item   If $S$ is a surface over an algebraically closed field $k$, and $\alpha\in\Br(k(S))$, then de Jong~\cite{de_jong_period_2004} showed that $\per(\alpha)=\ind(\alpha)$.
    \item   If $S$ is a surface over $\FF_q$, and if $\alpha\in\Br(\FF_q(S))$, then $\ind(\alpha)|\per(\alpha)^2$, by Lieblich~\cite{lieblich_period_2009}.
    \item   If $C$ is a curve over $\QQ_p$, and if $\alpha\in\Br(\QQ_p(C))$ has period prime
        to $p$, then $\ind(\alpha)|\per(\alpha)^2$, by Saltman~\cite{saltman_division_1997}.
\end{itemize}

A few other results are known, but higher-dimensional cases remain elusive.

\begin{problem}[Period--index]
    Fix $X$, and find an integer $e_r$ such that
    \begin{equation*}
        \ind(\alpha)|\per(\alpha)^{e_r}
    \end{equation*}
    for all $\alpha\in{_{r^{\infty}}\Br(X)}$, where $_{r^{\infty}}\Br(X)$ denotes the $r$-primary part of the Brauer group.
\end{problem}

The main global result of which we are aware is for surfaces over algebraically closed fields, where $\per=\ind$ by Lieblich~\cite{lieblich_twisted_2008}.
Additionally, the above-mentioned argument of Saltman reduces the problem for smooth varieties over algebraically closed fields to their function fields.


Our theorems~\ref{thm:bound} and~\ref{thm:toperind} give a solution to this problem for finite CW-complexes. In dimension $4$ a stronger bound than ours is possible
via a direct argument using the Atiyah--Hirzebruch spectral sequence.

\section{Twisted $K$-theory}\label{sec:twistk}

In this section, we recall the twisted topological complex $K$-theory spectrum $\KU(X)_{\alpha}$ for any space $X$ and any class $\alpha\in\Hoh^3(X,\ZZ)_{\tors}$.
This spectrum is given by Atiyah--Segal~\cite[Section~4]{atiyah-segal}, generalizing work of Donovan and Karoubi~\cite{donovan_karoubi}.
For $X$ a CW-complex and $\alpha$ in the Brauer group, we introduce the twisted algebraic $K$-theory spectrum $\mathcal{K}^{\alpha}(X)$,
which is the algebraic $K$-theory of the category of $\alpha$-twisted topological vector bundles on $X$.
There is a map of spectra $\mathcal{K}^{\alpha}(X)\rightarrow \KU(X)_{\alpha}$, which induces an isomorphism $\mathcal{K}^{\alpha}_0(X)\riso \KU^0(X)_{\alpha}$
when $X$ is compact. When $X$ is a complex noetherian scheme, we obtain in this way
\begin{equation*}
    \K^{\alpha}(X)\rightarrow\mathcal{K}^{\alpha}(X)\rightarrow \KU(X)_{\alpha},
\end{equation*}
where $\K^{\alpha}(X)$ is the algebraic $K$-theory of locally free and finite rank $\alpha$-twisted $\mathcal{O}_X$-modules on $X$.
Since $\KU(X)_{\alpha}$ satisfies descent, the map above produces a map of spectra
\begin{equation*}
    a_{\et}\K^{\alpha}(X)\rightarrow \KU(X)_{\alpha},
\end{equation*}
where $a_{\et}\K^{\alpha}$ is the \'etale-sheafification of the presheaf of spectra $\K^{\alpha}$.

Having dealt with the foundational questions to our satisfaction, we twist the unit map
$\mathrm{S}\rightarrow\KU$
following~\cite{antieau_index_2009}. Finally, we use the rank map for $\alpha$-twisted $K$-theory to define
an approximation to the index of $\alpha$. If $X$ is a finite CW-complex, this approximation is
precisely the index, which allows us to use twisted $K$-theory to settle the prime-divisor
problem and, when combined with the twisted unit map, to provide estimates for the period--index
problem. These applications will be the content of the following two sections.

\subsection{Twists of $\KU$}

\begin{definition}
    Let $\alpha\in\Hoh^3(X,\ZZ)$. Let $\KU(X)_{\alpha}$ be the twisted $K$-theory spectrum constructed
    in Atiyah-Segal~\cite{atiyah-segal}, and let $\KU(\cdot)_{\alpha}$ be the contravariant functor on spaces mapping to $X$ given by
    \begin{equation*}
        V\mapsto \KU(V)_{\alpha|_V}.
    \end{equation*}
\end{definition}

Note that there is an ambiguity of sign in the construction of $\KU(X)_\alpha$.
We choose the following normalization: use the construction
that results in $d_3^{\alpha}(1)=\alpha$ in the Atiyah--Hirzebruch spectral sequence for twisted $\KU$-theory
(see~\cite[Proposition~4.6]{atiyah_twisted_2006}).
This sign ambiguity has no bearing on the period--index problem of this paper. The methods will always produce a positive integer, the index,
that is the same for $\alpha$ and $-\alpha$.

We use the following properties of $\KU(\cdot)_{\alpha}$.

\begin{proposition}\label{prop:descent}
    The presheaf of spectra $\KU(\cdot)_{\alpha}$ commutes with homotopy colimits (taken in spaces) of spaces over $X$. In particular,
    if $\mathcal{V}_*\rightarrow X$ is a hypercover of $X$, then
    \begin{equation*}
        \KU(X)_{\alpha}\rwe \holim_{\Delta}\KU(\mathcal{V}_*)_{\alpha}.
    \end{equation*}
    That is, $\KU(\cdot)_{\alpha}$ satisfies descent.
    \begin{proof}

      Suppose that $X$ is a CW-complex, and that $\{P_n\}_{n\in\ZZ}$ is a spectrum over $X$. It suffices here for this to mean simply that
      there are fibrations $\pi_n:P_n\rightarrow X$ with fixed distinguished sections and fixed weak equivalences $\Omega_X P_n\we_X
      P_{n-1}$, where the weak equivalence induces weak equivalences on the fibers over $X$. So, $P_n$ is an $\Omega$-spectrum over
      $X$. Atiyah and Segal~\cite[Section~4]{atiyah-segal} show that twisted $K$-theory is represented by just such a spectrum. To prove the
      proposition, it suffices to show that the sections of the spectrum over a space $Y\rightarrow X$ form an $\Omega$-spectrum, and to
      show that taking spaces of sections commutes with homotopy colimits.

      Now, suppose that $i\mapsto Y_i$ is an $I$-diagram of CW-complexes, each mapping to $X$, and suppose that $\hocolim_i Y_i\rightarrow
      X$ is a weak equivalence.  

      For an arbitrary CW-complex, $W$, let $\Map_\pi(W, P_n)$ denote the mapping spaces over $\pi$. It is the pull-back
        \begin{equation*}
          \xymatrix{ \Map_\pi(W,P_n) \ar[r] \ar[d]  & \ar[d]   \Map(W,P_n) \\
               \ast \ar[r] & \Map(W,X).}
        \end{equation*}
        which is also a homotopy pull-back since $P_n \to X$ is a fibration.
        There is a natural homotopy equivalence $\holim_i \Map (Y_i, Z) \to \Map (X, Z)$ for an arbitrary space $Z$; by application of
        $\holim_i (Y_i, \cdot )$ to the diagram
        \begin{equation*}
          \xymatrix{ & \ar[d] P_n \\ \ast \ar[r] & X}
        \end{equation*}
        and use of the Fubini theorem for homotopy limits, we deduce that
        \begin{equation*}
            \Map_\pi(X,P_n)\rightarrow\holim_i\Map_\pi(Y_i,P_n)
        \end{equation*}
        is a weak equivalence.
        Now it suffices to prove that \[\Map_\pi(X,\Omega_X P_n)\we\Omega\Map_\pi(X,P_n),\] which follows since $\Omega_X
        P_n$ is a limit.

        This implies that $\KU(\cdot)$ satisfies descent since $\hocolim_{\Delta}(V_{\ast})\we X$ by~\cite{dugger-isaksen}.
    \end{proof}
\end{proposition}

\begin{proposition}\label{p:descSS}
    Let $X$ be a topological space, and let $\alpha\in\Hoh^3(X,\ZZ)_{\tors}$. There is an Atiyah--Hirzebruch (or descent) spectral sequence:
    \begin{equation*}
        \Eoh_2^{p,q}=\Hoh^p(X,\ZZ(q/2))\Rightarrow \KU^{p+q}(X)_{\alpha}
    \end{equation*}
    with differentials $d_r^{\alpha}$ of degree $(r,-r+1)$,
    which converges strongly if $X$ is a finite dimensional CW-complex.
    The edge map $\KU^0(X)_{\alpha}\rightarrow\Hoh^0(X,\ZZ)$ is the index map (or rank map, or reduced norm map).
    \begin{proof}
        This is~\cite[Theorem~4.1]{atiyah-segal}.
    \end{proof}
\end{proposition}

We will actually use a re-indexed, Bousfield--Kan-style, version of this spectral sequence:
\begin{equation*}
    \Eoh_2^{p,q}=\Hoh^p(X,\pi_q\KU)\Rightarrow\pi_{q-p}\KU(X)_{\alpha}.
\end{equation*}
The differentials are $d_r^{\alpha}$ of degree $(r,r-1)$.

\begin{proposition}
    The differential
    \begin{equation*}
        d_3^{\alpha}:\Hoh^0(X,\ZZ)\rightarrow\Hoh^3(X,\ZZ)
    \end{equation*}
    sends $1$ to $\alpha$.
    \begin{proof}
        This is~\cite[Proposition~4.6]{atiyah_twisted_2006}, where we have altered the construction to change the sign.
        See also Antieau~\cite{antieau_cech_2010} for the analogous computation in twisted algebraic $K$-theory.
    \end{proof}
\end{proposition}

\subsection{Twisted algebraic $K$-theory}

Throughout, if $\mathcal{A}$ is an Azumaya algebra over $X$, an $\mathcal{A}$-module will mean a left $\mathcal{A}$-module whose total
space is a finite dimensional complex topological vector bundle over $X$. The category of $\mathcal{A}$-modules and $\mathcal{A}$-module homomorphisms
will be denoted $\Vect^{\mathcal{A}}$. The category $\Vect^{\mathcal{A}}$ is a topological category with direct sum, so by Segal~\cite{segal-categories} it has an algebraic $K$-theory
spectrum $\mathcal{K}^{\mathcal{A}}(X)$.

If $\alpha\in\Hoh^3(X,\ZZ)_{\tors}$, there is also a category $\Vect^{\alpha}$ of $\alpha$-twisted finite dimensional complex vector bundles (see for instance~\cite{caldararu_derived_2000}
or~\cite{karoubi-twisted}).
Let $\mathcal{K}^{\alpha}(X)$ be the $K$-theory of this topological category with direct sum.
If $\mathcal{E}$ is an $\alpha$-twisted vector bundle, then its sheaf of endomorphisms is an
Azumaya algebra with class $\alpha$. The following statement gives the converse.

\begin{proposition}\label{prop:azsheafcor}
    If $\mathcal{A}$ is an Azumaya algebra of degree $n$ with class $\alpha$, then $\mathcal{A}\iso\End(\mathcal{E})$ for some $\alpha$-twisted vector bundle $\mathcal{E}$ of rank $n$.
    The $\alpha$-twisted sheaf is unique up to tensoring with untwisted line bundles.
    \begin{proof}
        See~\cite[Theorem~1.3.5]{caldararu_derived_2000} or~\cite[Section~3]{karoubi-twisted}.
    \end{proof}
\end{proposition}

\begin{proposition}\label{prop:morita}
    Tensoring with $\mathcal{E}^*$, the dual of $\mathcal{E}$, induces an equivalence of categories $\Vect^{\alpha}\rightarrow\Vect^{\mathcal{A}}$.
    \begin{proof}
        See~\cite[Theorem~1.3.7]{caldararu_derived_2000} or~\cite[Section~3]{karoubi-twisted}.
    \end{proof}
\end{proposition}

\begin{proposition}
    If $\mathcal{A}$ and $\mathcal{B}$ are Brauer-equivalent Azumaya algebras over $X$, then $\Vect^{\mathcal{A}}$ and $\Vect^{\mathcal{B}}$ are equivalent categories.
    \begin{proof}
        This follows from the previous two propositions.
    \end{proof}
\end{proposition}

If $X$ is a finite CW-complex, the twisted topological $K$-group $\KU^0(X)_{\alpha}$ may be identified with the Grothendieck group of left $\mathcal{A}$-modules.

\begin{proposition}\label{prop:ku0compact}
    If $X$ is compact and Hausdorff, and if $\mathcal{A}$ is an Azumaya algebra on $X$ with class $\alpha$ (so that $\alpha$ is torsion),
    then $\mathcal{K}_0^{\mathcal{A}}(X)\iso\KU^0(X)_{\alpha}$. This isomorphism is uniquely defined up to the natural action of $\Hoh^2(X,\ZZ)$ on the left.
    \begin{proof}
        This follows from~\cite[Section~3.1]{atiyah-segal}. In fact, they show that $\KU^0(X)_{\alpha}$ is isomorphic to
        the Grothendieck group of finitely generated projective left
        $\Gamma(X,\mathcal{A})$-modules. The isomorphism is not unique, but is unique up to
        tensoring with line bundles.
        The proof of Swan's theorem~\cite{swan} extends to show that the categories of finitely generated projective left $\Gamma(X,\mathcal{A})$-modules and
        left $\mathcal{A}$-modules that are finite dimensional vector bundles are equivalent. See also Karoubi~\cite[Section~4, Section~8.3]{karoubi-twisted}.
    \end{proof}
\end{proposition}

\begin{corollary}
    If $X$ is a finite CW-complex and if $\mathcal{A}$ is an Azumaya algebra on $X$, there is a natural map
    $\mathcal{K}^{\mathcal{A}}(X)\rightarrow\KU(X)_{\alpha}$ inducing an isomorphism in degree zero. This map is uniquely defined up to the action of $\Hoh^2(X,\ZZ)$ on the left.
    \begin{proof}
        The isomorphism of the proposition is induced by a map from the classifying space of the topological category $\Vect^{\mathcal{A}}$ to the zero-space of $\KU(X)_{\alpha}$.
        See the proof of~\cite[Definition~3.4]{atiyah-segal}.
        By the standard adjunction between $\Gamma$-spaces and spectra in Segal~\cite{segal-categories}, this induces a map $\mathcal{K}^{\mathcal{A}}(X)\rightarrow\KU(X)_{\alpha}$.
    \end{proof}
\end{corollary}

\begin{proposition}\label{prop:alg}
    If $X$ is a countable CW-complex, and if $\mathcal{A}$ is an Azumaya algebra with class $\alpha\in\Br(X_{\topo})$, then there is a map of spectra
    \begin{equation*}
        \mathcal{K}^{\mathcal{A}}(X)\rightarrow \KU(X)_{\alpha}.
    \end{equation*}
    This map is uniquely defined up to the natural action of $\Hoh^2(X,\ZZ)$ on the left.
    \begin{proof}
        First we give the proof for countable $d$-dimensional CW-complexes $X$. Suppose that we have constructed maps $\mathcal{K}^{\mathcal{A}}(X)\rightarrow\KU(X)_{\alpha}$
        for all finite CW-subcomplexes of $X$ and all countable $(d-1)$-dimensional CW-subcomplexes of $X$. Let $X$ be a countable $d$-dimensional CW-complex, let $X_{d-1}$ be its $(d-1)$-skeleton,
        and let $$X_{d-1}=X_{d,0}\subseteq X_{d,1}\subseteq X_{d,2}\subseteq\cdots=X$$ be an inductive construction of $X$, where $X_{d,k}$ is constructed from $X_{d,k-1}$ by attaching a single cell.
        Assume we have constructed $\mathcal{K}^{\mathcal{A}}(X_{d,k})\rightarrow\KU(X_{d,k})_{\alpha}$. Then, $X_{d,k+1}$ is the homotopy pushout
        \begin{equation*}
            \begin{CD}
                S^{d-1} @>>>    X_{d,k}\\
                @VVV            @VVV\\
                D^d     @>>>    X_{d,k+1}.
            \end{CD}
        \end{equation*}
        Since $\KU(\cdot)_{\alpha}$ commutes with homotopy colimits, we get a homotopy pull-back diagram of spectra by proposition~\ref{prop:descent}
        \begin{equation*}
            \begin{CD}
                \KU(X_{d,k+1})_{\alpha} @>>>    \KU(D^d)_{\alpha}\\
                @VVV                            @VVV\\
                \KU(X_{d,k})_{\alpha}   @>>>    \KU(S^{d-1})_{\alpha}.
            \end{CD}
        \end{equation*}
        By the induction hypothesis, $K^{\mathcal{A}}(X_{d,k+1})$ fits into a homotopy commutative diagram
        \begin{equation*}
            \begin{CD}
                \mathcal{K}^{\mathcal{A}}(X_{d,k+1}) @>>>    \KU(D^d)_{\alpha}\\
                @VVV                            @VVV\\
                \KU(X_{d,k})_{\alpha}   @>>>    \KU(S^{d-1})_{\alpha}.
            \end{CD}
        \end{equation*}
        So, by the property of the homotopy pull-back, we get a map $$\mathcal{K}^{\mathcal{A}}(X_{d,k+1})\rightarrow\KU(X_{d,k+1})_{\alpha}.$$
        Inductively, we construct a map between sequences of spectra
        \begin{equation*}
            \begin{CD}
                \mathcal{K}^{\mathcal{A}}(X_{d,k})   @>>>    \KU(X_{d,k})_{\alpha}\\
                @VVV                                @VVV\\
                \mathcal{K}^{\mathcal{A}}(X_{d,k-1})@>>>   \KU(X_{d,k-1})_{\alpha}.
            \end{CD}
        \end{equation*}
        The spectrum $\mathcal{K}^{\mathcal{A}}(X)$ maps to the homotopy limit of the
        left-hand tower, and hence to the homotopy-limit of the right-hand tower, which is
        $\KU(X)_\alpha$ by proposition~\ref{prop:descent}.
        The case where $X$ is countable and infinite-dimensional is proved in the same way.
    \end{proof}
\end{proposition}

\subsection{Comparison maps}\label{sec:comparisons}

Let $X$ be a noetherian complex scheme, and let $\alpha\in\Br(X_{\et})$. Let $a_{\et}\K^{\alpha}$ be the \'etale sheafification
of the $\alpha$-twisted algebraic $K$-theory
presheaf $\K^{\alpha}$, as studied in Antieau~\cite{antieau_index_2009}.
Let $\KU(X)_{\alpha}$ be the twisted topological $K$-theory associated to the image of $\alpha$ by $\Br(X_{\et})\rightarrow\Br(X_{\topo})$.

\begin{theorem}\label{thm:comparison}
    There is a map of hypersheaves of spectra
    \begin{equation*}
        a_{\et}\K^{\alpha}\rightarrow \KU(\cdot)_{\alpha}
    \end{equation*}
    on the small \'etale site of $X$, and this map induces an isomorphism on $0$-homotopy sheaves $\ZZ\riso\ZZ$. This map is unique up to the action on the left of $\Pic(X)$.
    \begin{proof}
        Let $\mathcal{A}$ be an algebraic Azumaya algebra with class $\alpha\in\Br(X_{\et})$.
        Since any open subset of $X$ is triangulable, there is a map of presheaves of spectra on the small \'etale site of $X$
        \begin{equation*}
            \mathcal{K}^{\mathcal{A}}\rightarrow \KU(\cdot)_{\alpha}
        \end{equation*}
        by proposition~\ref{prop:alg},
        where $\mathcal{K}^{\mathcal{A}}(U)$ is the algebraic $K$-theory of topological left $\mathcal{A}$-modules which are vector bundles over $U$.
        There is a map of presheaves of spectra
        \begin{equation*}
            \K^{\mathcal{A}}\rightarrow\mathcal{K}^{\mathcal{A}},
        \end{equation*}
        where $\K^{\mathcal{A}}$ is the presheaf of $K$-theory spectra of algebraic $\mathcal{A}$-modules which are locally free and finite rank as coherent $\mathcal{O}_X$-modules.
        Since $\KU(\cdot)_{\alpha}$ satisfies descent for hypercovers by proposition~\ref{prop:descent}, it follows by~\cite{dugger_hypercovers_2004} that the map
        \begin{equation*}
            \K^{\mathcal{A}}\rightarrow\mathcal{K}^{\mathcal{A}}\rightarrow \KU(\cdot)_{\alpha}
        \end{equation*}
        factors through the sheafification $a_{\et}\K^{\mathcal{A}}$ of $\K^{\alpha}$. By~\cite{lieblich_twisted_2008}, \cite{caldararu_derived_2000}, or \cite{antieau_index_2009},
        there are natural weak equivalences of presheaves $\K^{\alpha}\we\K^{\mathcal{A}}$ unique up to an action of $\Pic(X)$ on the left. The theorem follows.
    \end{proof}
\end{theorem}

\subsection{Twisting units}\label{sec:twistingorientations}

This section is technical, but it provides the key tool for our main period--index result: the twisted unit-morphism.

Let $\sphere$ be the sphere spectrum. In \cite[Corollary~6.3.2.16]{lurie_higher_algebra}, drawing on work of \cite{ekmm}, \cite{hss}, and others,
Lurie constructs a symmetric monoidal category structure on $\Mod_{\sphere}$,
the $\infty$-category of $\sphere$-modules (or, equivalently, the $\infty$-category of spectra). The
associative algebras in $\Mod_\sphere$ are shown to model $A_{\infty}$-ring spectra, and the commutative algebras in $\Mod_{\sphere}$
are shown to model $E_{\infty}$-ring spectra (see the introduction to Section 7.1 of~\cite{lurie_higher_algebra}).

Let $\rsphere$ be a commutative $\sphere$-algebra (an algebra object in $\Mod_\sphere$), and let $\Line_{\rsphere}$ denote the $\infty$-groupoid
in $\Mod_{\rsphere}$ generated by $\rsphere$-modules equivalent to $\rsphere$. Finally, set $BGL_1 \rsphere=|\Line_{\rsphere}|$, the geometric realization. To a
map of spaces $f:X\rightarrow BGL_1 \rsphere$, one associates a homotopy class of maps of simplicial sets
$f:\Sing(X)\rightarrow\Line_{\rsphere}$, where $\Sing(X)$ is the simplicial set of simplices in $X$. This should be thought of as a locally-free rank one line bundle
over the constant sheaf $\rsphere$ on $X$.

Given $f:\Sing(X)\rightarrow\Line_{\rsphere}$, in~\cite{abghr}, an $\rsphere$-module $X^f$, called the Thom spectrum of $f$, is defined as
\begin{equation*}
    X^f=\colim(\Sing(X)\rightarrow\Line_{\rsphere}\rightarrow\Mod_{\rsphere}).
\end{equation*}
See also~\cite{abg}.
This is the colimit of the $\Sing(X)$-diagram $f$ in the $\infty$-category $\Mod_\rsphere$. See~\cite[Section~1.2.13 and Chapter~4]{lurie_htt}
for definitions and properties of these colimits. In~\cite[Section~4.2.4]{lurie_htt}, Lurie shows that these colimits agree with the usual notion of homotopy colimits
when the diagram takes values in the nerve of a simplicial model category.

\begin{definition}
The $f$-twist of $\rsphere$ is the internal dual to $X^f$:
\begin{equation*}
    \rsphere(X)_f=\Mod_{\rsphere}(X^f,\rsphere).
\end{equation*}
Let $\rsphere(\cdot)_f$ be the associated presheaf of spectra on $X$.
\end{definition}

\begin{example}
  The space $BGL_1\KU$ is equivalent to \[K(\ZZ/2,1)\times K(\ZZ,3)\times BB\SU\] and consequently classes in $\Hoh^3(X,\ZZ)$ give rise to
  twisted $K$-theory spectra.  It is explained in~\cite[Section~5]{abg} how this agrees with the Atiyah-Segal construction.
\end{example}

Let $\Shv_{\Sp}(X)^{\wedge}$ denote the $\infty$-category of hypersheaves of spectra on $X$. This category can be constructed
from presheaves of spectra on $X$ by taking the full subcategory of presheaves which are local with respect to $\infty$-connective
maps on $X$. See~\cite[Section~1]{dag7} and \cite[Lemma~6.5.2.12]{lurie_htt}. For another perspective,
see~\cite{jardine-symmetric}. These are related by~\cite[Proposition~6.5.2.14]{lurie_htt}.

\begin{example}
    By proposition~\ref{prop:descent}, $\KU(\cdot)_\alpha$ is a hypersheaf of spectra.
\end{example}

In general, if $u:\rsphere\rightarrow\tsphere$ is a map of commutative $\sphere$-algebras, and if $f:X\rightarrow BGL_1\tsphere$ is a map, it is not necessarily possible
to find an $\rsphere$-module $\msphere$ in $\Shv_{\Sp}(X)^{\wedge}$ together with a map $\msphere\rightarrow\rsphere(\cdot)_f$ restricting to $u$ on small open sets.
For instance, this is never the case for the unit $u:\sphere\rightarrow\KU$ and the non-trivial twists $\KU(\cdot)_\alpha$ for $\alpha\in\Hoh^3(X,\ZZ)$.

We show now, however, that if $\alpha\in\Hoh^3(X,\ZZ)$ is torsion, then there is a finite cover of $\sphere$, say $\tsphere$, a map $u:\tsphere\rightarrow\KU$ extending $\sphere\rightarrow\KU$,
and a twist, $u_{\beta}:\tsphere(\cdot)_\beta\rightarrow\KU(\cdot)_\alpha$ derived from a map $\beta:X\rightarrow BGL_1\tsphere$.

\begin{definition}
    Let $\sphere[\ZZ/r]$ denote the algebraic $K$-theory of the symmetric monoidal category $rSets$ of finite disjoint unions of $\ZZ/r$
    with the faithful $\ZZ/r$ action and $\ZZ/r$-equivariant maps. The spectrum $\sphere[\ZZ/r]$ is a commutative $\sphere$-algebra.
\end{definition}

\begin{lemma}\label{lem:msets}
    The homotopy groups of $\sphere[\ZZ/r]$ are
    \begin{equation*}
        \pi_q\sphere[\ZZ/r]\iso \pi_q^s\oplus\pi_q^s(B\ZZ/r),
    \end{equation*}
    where $B\ZZ/r$ is the classifying space of $\ZZ/r$. Moreover, $\pi_q^s(B\ZZ/r)$ maps surjectively onto the $r$-primary part of $\pi_q^s$.
    \begin{proof}
        The first statement follows from the Barratt-Kahn-Priddy-Quillen theorem (see Thomason~\cite[Lemma~2.5]{thomason_first_1982}),
        which says that $\sphere[\ZZ/m]$ is equivalent to the spectrum $\Sigma^{\infty}(B\ZZ/m)_+$. The second statement follows from the Kahn-Priddy theorem~\cite{kahn-priddy-1,kahn-priddy-2}.
    \end{proof}
\end{lemma}

\begin{lemma}\label{lem:munits}
    For any $r$, the unit map $\sphere\rightarrow\KU$ factors through $\sphere\rightarrow\sphere[\ZZ/r]$.
    And, the map $\sphere[\ZZ/r]\rightarrow\KU$ is a map of $\sphere$-algebras.
    \begin{proof}
        This follows by embedding $\ZZ/r$ into $\CC^*$ as the $r$th roots of unity and then applying Segal's construction~\cite{segal-categories}
        to obtain $\sphere\rightarrow\sphere[\ZZ/r]\rightarrow\ku$. Now, compose with $\ku\rightarrow\KU$.
        That $\KU$ is in fact a commutative $S$-algebra is~\cite[Theorem~4.3]{ekmm}.
    \end{proof}
\end{lemma}

%

\begin{proposition}
    There is a natural map
    \begin{equation*}
        K(\ZZ/r,2)\rightarrow BGL_1\sphere[\ZZ/r]
    \end{equation*}
    \begin{proof}
        The group $\ZZ/r$ acts naturally as monoidal auto-equivalences of the symmetric monoidal category $rSets$ and so $\ZZ/r$ acts naturally on $\sphere[\ZZ/r]$.
        This gives a map $B\ZZ/r\rightarrow GL_1\sphere[\ZZ/r]$. By delooping, we get the desired map.
    \end{proof}
\end{proposition}

\begin{proposition}\label{prop:twistedunit}
    Let $\alpha\in\Hoh^3(X,\ZZ)_{\tors}$ be $r$-torsion, and lift $\alpha$
    to $\sigma\in\Hoh^2(X,\ZZ/r)$. Then, there is a natural map
    \begin{equation*}
        \sphere[\ZZ/r](\cdot)_{\sigma}\rightarrow \KU(\cdot)_{\alpha}
    \end{equation*}
    of presheaves of spectra restricting locally to the map of lemma~\ref{lem:munits}.
    \begin{proof}
        It suffices to check that the deloopings
        \begin{equation*}
            \begin{CD}
                K(\ZZ/r,1)  @>>>    GL_1\sphere[\ZZ/r]\\
                @V\beta VV                @VVV\\
                K(\ZZ,2)    @>>>    GL_1 \KU
            \end{CD}
        \end{equation*}
        commute up to homotopy. To check this we can check on maps out of finite CW-complexes. Let $\gamma\in\Hoh^1(X,\ZZ/r)$ be a
        $\ZZ/r$-torsor, then the corresponding automorphism of the constant sheaf $\sphere[\ZZ/r]$ is given by tensoring with $\gamma$. The
        induced automorphism of the constant sheaf $\KU$ is tensoring with the complex line bundle induced by $\gamma$. On the other hand
        $\Hoh^1(X,\ZZ/r)\rightarrow\Hoh^2(X,\ZZ)$ sends $\gamma$ to $c_1(\mathcal{L})$, where $\mathcal{L}$ is the complex line bundle
        associated to $\gamma$. The corresponding automorphism of $\KU$ is given by tensoring with $\mathcal{L}$, by
        construction~\cite{abg}.
    \end{proof}
\end{proposition}

\begin{corollary}\label{cor:morss}
    Let $\alpha\in\Hoh^3(X,\ZZ)_{\tors}$ be $r$-torsion, and lift $\alpha$ to $\sigma\in\Hoh^2(X,\ZZ/r)$. Then there is a spectral sequence
    \begin{equation*}
        \Eoh^2_{p,q}=\Hoh^p(X,\pi_q\sphere[\ZZ/r])\Rightarrow\pi_{q-p}\sphere[\ZZ/r]^{\wedge}(X)_{\sigma},
    \end{equation*}
    where $\sphere[\ZZ/r]^{\wedge}(\cdot)_{\sigma}$ is the hypersheaf associated to the
    presheaf $\sphere[\ZZ/r](\cdot)_\sigma$. The differentials $d_k^{\sigma}$
    are of degree $(k,k-1)$. There is a morphism of spectral sequences
    \begin{equation*}
        \Big(\Hoh^p(X,\pi_q\sphere[\ZZ/r])\Rightarrow\pi_{q-p}\sphere[\ZZ/r]^{\wedge}(X)_{\sigma}\Big)\rightarrow\Big(\Hoh^p(X,\pi_q\KU)\Rightarrow\pi_{q-p}\KU(X)_{\alpha}\Big).
    \end{equation*}
    \begin{proof}
        This is the standard descent spectral sequence. Note that
        $\sphere[\ZZ/r](\cdot)_\sigma\rightarrow\KU(\cdot)_\alpha$ factors through hypersheafification
        $\sphere[\ZZ/r](\cdot)_\sigma\rightarrow\sphere[\ZZ/r]^{\wedge}(\cdot)_\sigma$ because $\KU(\cdot)_\alpha$ is a hypersheaf by proposition~\ref{prop:descent}.
    \end{proof}
\end{corollary}


\subsection{The $K$-theoretic index}

There are rank maps forming a commutative diagram
\begin{equation*}
    \begin{CD}
        \mathcal{K}_0^{\alpha}(X)   @>>>    \ZZ\\
        @VVV                                @|\\
        \KU^0(X)_{\alpha}            @>>>    \ZZ
    \end{CD}
\end{equation*}
The index $\ind(\alpha)$ is the positive generator of the image of the top arrow, and we define the \textit{$\K$-index\/} $\ind_{\K}(\alpha)$
to be the positive generator of the image of the bottom arrow. We may view the rank map as the pull-back of an $\alpha$-twisted virtual
bundle to a point, that is to say the bottom map may be written as $\KU^0(X)_\alpha \to\KU^0(\ast)_0 \iso \ZZ$.
We may compare the Atiyah--Hirzebruch spectral sequences for the twisted $\KU$-theory of $X$ and $\ast$ to arrive at the following conclusion.

\begin{proposition} \label{p:indKasPermCyc}
  Let $X$ be a finite-dimensional connected CW-complex and let $\alpha \in \Hoh^3(X, \ZZ)_{\tors}$. Then $\ind_\K(\alpha)$ is the positive generator of
  the subgroup  $\Eoh_\infty^{0,0} $ of the group $ \Eoh_2^{0,0} = H^0(X, \ZZ) = \ZZ$ in the Atiyah--Hirzebruch spectral sequence of proposition \ref{p:descSS}.
\end{proposition}

Note that under $\mathcal{K}^{\mathcal{A}}_0(X)\iso\mathcal{K}^{\alpha}_0(X)$, the map $\mathcal{K}^{\mathcal{A}}_0(X)\rightarrow\ZZ$ is the
reduced norm map. It sends an $\mathcal{A}$-module to the quotient of its rank as a complex vector bundle by the degree of $\mathcal{A}$.

\begin{lemma}\label{lem:indkind}
    For any $\alpha\in\Hoh^3(X,\ZZ)_{\tors}$,
    \begin{equation*}
        \per(\alpha)|\ind_{\K}(\alpha)|\ind(\alpha).
    \end{equation*}
    Moreover, $ind(\alpha)$ is equal to
    \begin{equation*}
        \gcd\{\deg(\mathcal{A}) : [\mathcal{A}]=\alpha\}.
    \end{equation*}
    \begin{proof}
        The statement about divisibilities follows from the commutative diagram.
        The second statement follows from propositions~\ref{prop:azsheafcor} and~\ref{prop:morita}.
    \end{proof}
\end{lemma}

\begin{lemma}
    If $X$ is a finite CW-complex, then $\ind_{\K}(\alpha)=\ind(\alpha)$.
    \begin{proof}
        This follows immediately from proposition~\ref{prop:ku0compact}.
    \end{proof}
\end{lemma}

\begin{example}
    Suppose that $X$ is a finite CW-complex of dimension no greater than $6$, and suppose that $H^5(X,\ZZ)$ is torsion-free. Let $\alpha\in\Hoh^3(X,\ZZ)_{\tors}$. Then $\ind(\alpha)=\ind_{\K}(\alpha)=\per(\alpha)$.
    It suffices to compute $d_5^{\alpha}(\per(\alpha))$. This lands in a subquotient of $H^5(X,\ZZ)$. The differential
    \begin{equation*}
        d_3^{\alpha}:\Hoh^2(X,\ZZ)\rightarrow\Hoh^5(X,\ZZ)
    \end{equation*}
    is zero. Indeed, the differential is torsion since $\alpha$ is torsion, so it has torsion image, whereas $\Hoh^5(X,\ZZ)$ is non-torsion. As there are no differentials leaving $\Hoh^5(X,\ZZ)$,
    $\Eoh_5^{5,-4}=\Hoh^5(X,\ZZ)$.
    The element $d_5^{\alpha}(\per(\alpha))$ is some element of $\Hoh^5(X,\ZZ)$. It must be a torsion element since
    $\ind_{\K}(\alpha)=\ind(\alpha)$ is finite as $X$ is a finite CW-complex, and consequently it is zero. This gives $\ind(\alpha)=\ind_{\K}(\alpha)=\per(\alpha)$.
\end{example}

In the algebraic case, there is an analogue of the intermediary $\ind_{\K}$ that sits between $\per$ and $\ind$.

\begin{definition}
    The \'etale index $\mathrm{eti}(\alpha)$ of a class $\alpha\in\Br'(X_{\et})$ was defined in~\cite{antieau_index_2009} as the positive generator of the rank map $a_{\et}\K^{\alpha}_0(X)\rightarrow \ZZ$.
\end{definition}

\begin{corollary}
    For $X$ a scheme and $\alpha\in\Br(X_{\et})$,
    \begin{equation*}
        \ind_{\K}(\alpha_{\topo})|\mathrm{eti}(\alpha).
    \end{equation*}
    \begin{proof}
        This follows from theorem~\ref{thm:comparison}.
    \end{proof}
\end{corollary}

\section{The prime-divisor problem for topological spaces}\label{s:primediv}

In this short section we prove that for a finite CW complex $X$ and for $\alpha \in \Br(X)$, the
primes dividing $\per(\alpha)$ and $\ind(\alpha)$ agree. 

Let $A$ be a finitely-generated abelian group and $r$ a positive integer, not necessarily prime. We say $a \in A$ is $r$-primary if $r^ma=0$
for some positive integer $m$. We say $A$ is $r$-primary if all the elements of $A$ are $r$-primary. For a given positive integer $r$, the
class of $r$-primary finitely-generated abelian groups is a Serre class; it is closed under taking subobjects and the formation of
quotients and extensions. We denote this Serre class by $\mathcal{C}_r$. There is a mod-$\mathcal{C}_r$ Hurewicz theorem, \cite{serre_1958},
which implies that if $X$ is a simply-connected CW-complex and if every group $\pi_n(X)$ with $n>1$ is $r$-primary, then so too is every
homology group $\Hoh_n(X,\ZZ)$ with $n>0$. In particular, using the universal coefficients theorem, we know that $\Hoh^n(K(\ZZ/r,2), \ZZ)$
is $r$-primary for all $n>0$.

Since $K(\ZZ/r,2)$ is the universal space for $r$-torsion classes in $\Hoh^3( \cdot, \ZZ)$, we can exploit the $r$-primary nature of the
cohomology here to give a general statement concerning the differentials in the Atiyah--Hirzebruch spectral sequence of proposition
\ref{p:indKasPermCyc}. 

\begin{theorem} \label{t:PrimesPerIndAgree}
    If $X$ is a CW-complex and if $\alpha\in\Br(X)$, then $\ind_{\K}(\alpha)$ and $\per(\alpha)$ have the same prime divisors.
    \begin{proof}
      Since $\per(\alpha) | \ind_\K(\alpha)$, it suffices to show that any prime divisor of $\ind_\K(\alpha)$ is also a prime divisor of $\per(\alpha)$.
      
        That $\ind_{\K}(\alpha)$ is finite follows from the fact that $\alpha$ is in the Brauer group. Let $\per(\alpha)=r$ so that $\alpha$
        is the image under the unreduced Bockstein map of a class $\sigma\in\Hoh^2(X,\ZZ/r)$. In particular, $\alpha \in \Hoh^3(X, \ZZ)$
        is the pull-back of $\beta \in \Hoh^3(K(\ZZ/r, 2), \ZZ)$ under $\sigma$.

        We recall from proposition \ref{p:indKasPermCyc} that $\ind_\K(\alpha)$ is a
        generator for the group of permanent cycles in
        $\Eoh_2^{0,0}$. The map $\sigma$ induces a map of Atiyah--Hirzebruch spectral
        sequences, which we denote by $\Sigma$.
        \begin{equation*}
          \xymatrix@C=5pt{
            \IEoh^{p,q}_2 = \Hoh^p(X, \ZZ(q/2)) \ar@{=>}[rr] & & \KU^{p+q}(X)_\alpha \\
            \IIEoh^{p,q}_2 = \Hoh^p(K(\ZZ/r,2), \ZZ(q/2)) \ar@{=>}[rr] & \ar^{\Sigma}[u] & \KU^{p+q}(K(\ZZ/r,2))_\beta .}
        \end{equation*}
        We observe that $\Sigma : \IIEoh^{0,0}_2 = \ZZ \to \IEoh^{0,0}_2 = \ZZ$ is simply the isomorphism on $\Hoh^0$ induced by
        $\phi$. Subsequently, we find $\IEoh^{0,0}_n \subset \IEoh^{0,0}_2$ and similarly $\IIEoh^{0,0}_n \subset \IIEoh^{0,0}_2$; for
        dimensional reasons there are no differentials whose target is $\Eoh^{0,0}_n$ in either spectral sequence.
        We are therefore justified in writing
        \begin{equation}
          \label{eq:defmmprime}
          \IIEoh^{0,0}_n = m'_n \ZZ \subset \IEoh^{0,0}_n = m_n \ZZ \subset \IEoh^{0,0}_2 = \ZZ
        \end{equation}
        where $m_n$ and $m'_n$ are chosen to be nonnegative. 

        We claim that the prime factors of the integer $m'_n$ are prime factors of $r$. By definition we have exact
        sequences
        \begin{equation*}
          \xymatrix{ 0 \ar[r] & \IIEoh_{n+1}^{0,0} = m'_{n+1} \ZZ \ar[r] & \IIEoh_n^{0,0} = m'_n \ZZ \ar^{d_n}[r] & \IIEoh_n^{n, -n +1} }
        \end{equation*}
        Here $\IIEoh_n^{n,-n+1} $ is a subquotient of $\IIEoh_2^{n,-n+1} = \Hoh^n(K(\ZZ/r, 2), \ZZ(\frac{-n+1}{2}))$, and since the latter is
        $r$-primary, so is the former. In particular, the image of $d_n$ in the above sequence is $r$-primary, but this image is isomorphic
        to the group $\ZZ/(m'_{n+1}/m'_n)$, and consequently $m'_{n+1}/m'_n$ is a positive integer whose prime factors all divide $r$. The
        claim follows by induction.

        From equation \eqref{eq:defmmprime}, we conclude that $m_n | m'_n$ and so the prime factors of $m_n$ number among the prime factors
        of $r$. Finally, we observe that since $\ind_{\K}(\alpha)$ is finite we have a nonzero group on the $\Eoh_\infty$-page, that is
        \begin{equation*}
                 0 \subsetneq \IEoh_{\infty}^{0,0} \iso \ind_\K(\alpha) \ZZ \subset \IEoh_2^{0,0} = \ZZ
        \end{equation*}
        It follows that $\ind_\K(\alpha) = m_N$ for some sufficiently large $N$, and so we conclude that the prime numbers dividing
        $\ind_\K(\alpha)$ also divide $r = \per(\alpha)$, as required.
    \end{proof}
\end{theorem}

\begin{corollary} \label{c:indPerPrimesAgree}
  If $X$ is a finite CW-complex, and if $\alpha \in \Br(X)$, then the primes dividing $\per(\alpha)$
  and $\ind(\alpha)$ agree.
\end{corollary}
\begin{proof}
  This follows immediately from the theorem, since $\ind_\K(\alpha) = \ind(\alpha)$.
\end{proof}

\section{The period--index problem for topological spaces}\label{sec:perind}

The following theorem provides our main upper-bound.

\begin{theorem}\label{thm:bound}
    Suppose that $X$ is a $d$-dimensional CW-complex, and let $$\alpha\in\Hoh^3(X,\ZZ)_{\tors}.$$ Then, $\ind_{\K}(\alpha)$ is finite, and
    \begin{equation*}
        \ind_{\K}(\alpha)|\prod_{j=1}^{d-1}e_j^{\alpha}
    \end{equation*}
    where $e_j^{\alpha}$ is the exponent of $\pi_j^s(B\ZZ/\per(\alpha))$, where $\pi_j^s$ denotes stable homotopy.
    \begin{proof}
        Let $\sigma$ be a lift of $\alpha$ to $\Hoh^2(X,\ZZ/r)$, where $r=\per(\alpha)$. Then, by proposition~\ref{prop:twistedunit}, there is a
        map of presheaves of spectra $\sphere[\ZZ/r](X)_{\sigma}\rightarrow\KU(X)_{\alpha}$.
        Consider the spectral sequence of corollary~\ref{cor:morss}
        \begin{equation*}
            \Eoh_2^{p,q}=\Hoh^p(X,\pi_q\sphere[\ZZ/r])\Rightarrow\pi_{q-p}\sphere[\ZZ/r]^{\wedge}(X)_{\sigma}.
        \end{equation*}
        The differentials $d_k^\sigma$ are of degree $(k,k-1)$. The groups $\pi_q\sphere[\ZZ/r]$ are $\pi_q^s\oplus\pi_q^s B\ZZ/r$ by
        lemma~\ref{lem:msets}. Write $\ell_q^{\alpha}$ for the exponent of $\pi_q\sphere[\ZZ/r]$.  Because the dimension of $X$ imposes an
        upper bound on where we may find nonzero cohomology groups, the last possible non-zero differential coming from $\Eoh^{0,0}$ is
        $d_d$. Therefore we are concerned with $\pi_q\sphere[\ZZ/r]$ only up to $q=d-1$. The differential $d_k^{\sigma}$ leaving
        $\Eoh_r^{0,0}$ lands in a group of exponent at most $\ell_{k-1}^{\alpha}$. Therefore,
        \begin{equation*}
            \prod_{j=1}^{d-1}\ell_j^{\alpha}
        \end{equation*}
        is a permanent cycle in the twisted spectral sequence for
        $\sphere[\ZZ/m]^{\wedge}(X)_{\sigma}$. Since there is
        a morphism of the spectral sequence for $\sphere[\ZZ/m]^{\wedge}(X)_\sigma$ to $\KU(X)_{\alpha}$, it follows that the product
        is a permanent cycle in the twisted spectral sequence for $\KU(X)_{\alpha}$.
        By theorem~\ref{t:PrimesPerIndAgree}, $\ind_{\K}(\alpha)$ divides the $r$-primary part of
        \begin{equation*}
            \prod_{j=1}^{d-1}\ell_j^{\alpha}.
        \end{equation*}
        By the second part of lemma~\ref{lem:msets}, the $r$-primary part of $\ell_j^{\alpha}$ is the exponent of $\pi_j^s(B\ZZ/r)$.
    \end{proof}
\end{theorem}

\begin{corollary}
    If $X$ is a finite CW-complex of dimension $d$, then theorem~\ref{thm:bound} holds with $\ind(\alpha)$ in place of $\ind_{\K}(\alpha)$.
    \begin{proof}
        This follows from lemma~\ref{lem:indkind}.
    \end{proof}
\end{corollary}

\begin{corollary}
    Let $X$ be a Stein space having the homotopy type of a finite CW-complex. Then
    theorem~\ref{thm:bound} holds with the analytic index in place of $\ind_{\K}(\alpha)$.
    \begin{proof}
        This follows from the Oka principle~\cite{grauert_remmert}.
    \end{proof}
\end{corollary}

We now analyze the integers $e_j^{\alpha}$ in a certain range in the particular case where $\per(\alpha)$ is a prime-power $\ell^n$.

\begin{proposition}
    Let $0<k<2\ell-3$. Then the $\ell$-primary component $\pi_k^s(\ell)$ of $\pi_k^s$ is zero, and also
    \begin{equation*}
        \pi_{2l-3}^s(\ell)=\ZZ/\ell.
    \end{equation*}
    \begin{proof}
        A proof of this can be found in~\cite[Theorems~1.1.13-14]{ravenel_complex_1986}.
    \end{proof}
\end{proposition}

\begin{proposition}
    For $0<k<2\ell-2$, the stable homotopy group $\pi_k^s(B\ZZ/\ell^n)$ is isomorphic to $\ZZ/\ell^n$ for $k$ odd and zero for $k$ even.
    \begin{proof}
        This is~\cite[Proposition~4.2]{antieau_index_2009}.
    \end{proof}
\end{proposition}

\begin{theorem}\label{thm:toperind}
    Let $X$ be a $d$-dimensional CW-complex, let $\ell$ be a prime such that $2\ell>d+1$, and let $\alpha\in\Br'(X)=\Hoh^3(X,\ZZ)_{\tors}$ satisfy
    $\per(\alpha)=\ell^k$; then
    \begin{equation*}
        \ind_{\K}(\alpha)|\per(\alpha)^{[\frac{d}{2}]}.
    \end{equation*}
    \begin{proof}
        This follows immediately from theorem~\ref{thm:bound} and the previous two propositions.
    \end{proof}
\end{theorem}

\begin{corollary}
    Let $X$ be a finite CW-complex of dimension $d$, let $\ell$ be a prime such that $2\ell>d+1$, and suppose $\alpha\in\Br(X)=\Hoh^3(X,\ZZ)_{\tors}$ satisfies
    $\per(\alpha)=\ell^k$; then,
    \begin{equation*}
        \ind(\alpha)|\per(\alpha)^{[\frac{d}{2}]}.
    \end{equation*}
\end{corollary}

\begin{corollary}
    Let $X$ be a $d$-dimensional Stein space having the homotopy type of a finite CW-complex, and let $\ell$ be a prime such that $2\ell>d+1$.
    If $\alpha\in\Br(X)$ satisfies $\per(\alpha)=\ell^k$; then,
    \begin{equation*}
        \ind(\alpha_{\an})|\per(\alpha)^{[\frac{d}{2}]},
    \end{equation*}
    where $\ind(\alpha_{\an})$ is the $\ell$-part of the greatest common divisor of the degrees of all analytic Azumaya algebras in the class $\alpha$.
\end{corollary}

\begin{remark}
    The Atiyah--Hirzebruch spectral sequence says that for $X$ a finite CW-complex of dimension at most $4$, $\per=\ind$. Therefore, the theorem
    is not sharp in general.
\end{remark}

\section{Lower Bounds on the Index}\label{sec:cohomology}

We first establish a number-theoretic result that will apply in the study of the cohomology of $PU_n$. We then consider the problem of
finding a degree-$n$ representative for a class in $\Br'(X) = \Hoh^3(X,\Z)_{\tors}$, which is the same as a factorization of $X \to K(\Z,3)$
as $X \to BPU_n \to K(\Z,3)$. We obtain a family of obstructions to such a factorization that can most easily be computed after application
of the reduced loopspace functor, $\Omega(\cdot)$, and we then use this family to furnish examples.

\subsection{Calculations in elementary number-theory}

\begin{definition}
  We define an integer-valued function $m:\NN \times \NN \to \NN$ by
  \begin{equation*}
    m(a,s) = \gcd \left\{ \binom{a}{i} \right\}_{i=1}^s
  \end{equation*}
  where the binomial coefficient $\binom{a}{i} = 0$ if $i >a $.
\end{definition}
We observe that $m(a,s)$ is a decreasing function of $s$, in that $m(a,s+1) | m(a,s)$ for all $s$, it begins with $m(a,1) = a$ and
stabilizes at $m(a,a) =1$.

In the following lemma and subsequently, the notation $[x]$ will be used to denote the integral part of the real number $x$.

\begin{lemma}
  Let $p$ be a prime number, and let $n$, $s$ be positive integers. Write $c=\max\{n-[\log_p s],
  0\}$, then $m(p^n,s) = p^c$.
\end{lemma}
\begin{proof}
  It suffices to determine the power of $p$ dividing $m(p^n, s)$.
  
  A theorem of Kummer \cite{Kummer_1852} says that the exponent of $p$ dividing $\binom{a +b}{b}$ is equal to the number of carries that
  arise in the addition of $a$ and $b$ in base-$p$ arithmetic. From this, we deduce that the exponent of $p$ dividing $\binom{p^n}{k}$ is at
  least $n-[\log_p k]$, although it may be more. For $k = p^j$, where $j \le n$, Kummer's result says this inequality is in fact an
  equality. Since the exponent of $p$ dividing $m(p^n, s)$ is the infimum of the exponents of $p$ dividing $\{\binom{p^n}{k} \}_{k=1}^s$,
  the stated result follows.
\end{proof}

\begin{lemma}
  Let $p$ be a prime number, and let $a$ be a positive integer relatively prime to $p$. Let $s$, $n$
  be  positive integers. Then the power of $p$ dividing $m(ap^n, s)$ coincides with the power of $p$
  dividing $m(p^n,s)$.
\end{lemma}
\begin{proof}
  It suffices to prove this for $s \le p^n$, because $m(p^n, p^n) = 1$, and $m(ap^n, s+1) | m(ap^n,
  s)$ for all $s$, so that if the result holds for $s= p^n$, it will hold trivially for $s > p^n$.

  Assume therefore that $s \le p^n$. We claim that the power of $p$ dividing $\binom{ap^n}{k}$ coincides with that
  dividing $\binom{p^n}{k}$ for all $k \le s$. If $0 \le r < k$, then $ap^n - r$ is not
  divisible by $p^{n+1}$, for otherwise we should have $r$ divisible by $p^n$, which can happen
  within the constraint $r < p^n$ only if $r =0$, in which case $ap^n - r = ap^n$ which is not divisible
  by $p^{n+1}$ by assumption.

  In the binomial expansion
  \begin{equation*}
    \binom{ ap^n}{k} = \frac{ (ap^n)(ap^n -1) \dots (ap^n - k +1) }{k!} 
  \end{equation*}
  the power of $p$ dividing a term $ap^n - r$ on the top row is no greater than $p^n$, and therefore
  coincides with the power of $p$ dividing $r$, and also with the power of $p$ dividing
  $p^n-r$. Comparing, term-by-term, we see that the power of $p$ dividing $\binom{ ap^n}{k}$
  agrees with that dividing
  \begin{equation*}
    \binom{p^n}{k} = \frac{ p^n(p^n-1) \dots (p^n - k+1)}{k!}
  \end{equation*}
  It follows that the power of $p$ dividing $m(ap^n, s)$ is the same as that dividing $m(p^n, s)$,
  as claimed.
\end{proof}

\begin{corollary} \label{c:descofm}
  For a positive integer $a$, with prime factorization $a = p_1^{n_1} p_2^{n_2} \dots p_r^{n_r}$,
  and for a given positive integer $s$, we have
  \begin{equation*}
    m(a,s) = \prod_{i=1}^r p_i^{ \max \{ n_i - [\log_{p_i} s], 0 \} }
  \end{equation*}
\end{corollary}
\begin{proof}
  Once one observes that the only prime factors $m(a,s)$ may have are the primes dividing $m(a,1) =
  a$, this result follows immediately from the previous two lemmas.
\end{proof}

\begin{definition}
  Let $b, s$ be positive integers. Define $n(b,s)$ as follows. Write $b = \prod_{i=1}^r p_i^{n_i}$, where the $p_i$ are unique primes and the $n_i$
  are positive integers. Define
  \begin{equation*}
    n(b,s) = \prod_{i=1}^r p_i^{n_i + [\log_{p_i} s]}
  \end{equation*}
\end{definition}

\begin{corollary} \label{c:useofn}
  Let $b$, $a$ and $s$ be positive integers. If $b \mid m(a,s)$, then $n(b, s) \mid a$.
\end{corollary}
\begin{proof}
  Write $b = \prod_{i=1}^r p_i^{n_i}$ as in the definition of $n(b,s)$. Write $a = a'\prod_{i=1}^r p_i^{n_i'}$, where $a'$ is relatively prime to $b$
  and where the $n_i'$ are not necessarily positive. Then $p_i^{n_i} \mid p_i^{n_i' - \max\{ 0, [\log_{p_i} s]\}}$, so that in particular $n_i +
  [\log_{p_i} s] < n_i'$, from which the claim follows.
\end{proof}

Note that for fixed $b$, $\lim_{s \to \infty} n(b,s)  = \infty$.

\subsection{Obstructions arising from the cohomology suspension}\label{sec:4.2}

In this section we obtain obstructions to factorizations $X \to BPU_a \to K(\ZZ,3)$ of $a$-torsion
classes by finding obstructions that survive an application of the loop-space functor $\Omega(X) \to
PU_a \to K(\ZZ,2)$. Not only is the integral cohomology of $PU_a$ easier to describe than
that of $BPU_a$, it also happens that the product structure on the cohomology of $PU_a$ 
yields a number of obstructions without our having to resort to cohomology operations.

To this end, we first describe what we need of the cohomology of $PU_a$. Thereafter the description
of lower bounds on the index in theorem \ref{t:ms5} is straightforward, and it is easy to use these
bounds to give examples where the index is arbitrarily large compared to the period.

For an element $\zeta$ in an abelian group $H$, let $\ord(\zeta)$ denote the order of $\zeta$ in $H$.

\begin{proposition} \label{p:cohoPUn}
  With integer coefficients, $\Hoh^1(PU_a, \ZZ) = 0$ and $\Hoh^2(PU_a, \ZZ) = \ZZ/a$. Fix a
  generator, $\eta \in \Hoh^2(PU_a, \ZZ)$.
    The following identity holds:
    \begin{equation*}
        \ord(\eta^s)=m(a,s),
    \end{equation*}
    In particular, $\ord(\eta^a)=1$, so $\eta^a=0$.
    \begin{proof}
        Consider the fibration
        \begin{equation*}
            U_a\rightarrow PU_a\rightarrow BS^1.
        \end{equation*}
        The integral cohomology of $U_a$ is
        \begin{equation*}
            \Hoh^*(U_a,\ZZ)=\Lambda_{\ZZ}(\alpha_1,\ldots,\alpha_a),
        \end{equation*}
        an exterior algebra in $a$ generators with the degree of $\alpha_s$ being $2s-1$, while
        \begin{equation*}
            \Hoh^*(BS^1,\ZZ)=\ZZ[\theta],
        \end{equation*}
        where the degree of $\theta$ is $2$. Consider the spectral sequence of the fibration
        \begin{equation*}
            \Hoh^p(BS^1,\Hoh^q(U_a,\ZZ))\Rightarrow\Hoh^{p+q}(PU_a,\ZZ).
        \end{equation*}
        Then \cite[Theorem~4.1]{baum-browder}
        says that $d_{2s}(\alpha_s)={a\choose s}\theta^s$. By induction and the algebra structure
        on the spectral sequence, it follows that the class $\eta^s$, which is the image of
        $\theta^s$, has
        order as stated. It also follows from the spectral sequence that $\Hoh^1(PU_a,\ZZ)=0$ and $\Hoh^2(PU_a, \ZZ) =
        (\ZZ/a)\eta$, as claimed.
    \end{proof}
\end{proposition}

Recall that if $\alpha \in \rHoh^n(X, A)$, then $\alpha$ may be represented as a based map $X \to K(A,
n)$. Applying the reduced loopspace functor, $\Omega(\cdot)$, one obtains a natural transformation
of functors
\begin{equation*}
  \Omega: \rHoh^n( \cdot, A) \to \rHoh^{n-1}(\Omega(\cdot), A)
\end{equation*}
This natural transformation is termed the \textit{cohomology suspension\/}.

From the Serre spectral sequence associated with the fibration $PU_a \to EPU_a \to BPU_a$ it follows that $\Hoh^3(BPU_a, \ZZ)$ is isomorphic
to the group $\ZZ/a$, generated by a class $\tilde{\eta}$, which may be chosen in such a way that $\Omega(\tilde{\eta}) : \Omega(BPU_a) \to
\Omega(K(\ZZ,3))$ is the distinguished generator $\eta$.

\begin{theorem} \label{t:ms5}
    Let $\tilde{\alpha} \in\Hoh^3(X,\ZZ)$ be a cohomology class, and let $\alpha\in\Hoh^2(\Omega
    X,\ZZ)$ be the cohomology suspension of this class. If $\tilde{\alpha}:X\rightarrow K(\ZZ,3)$ factors through $\tilde{\eta}: BPU_a\rightarrow K(\ZZ,3)$
    then $\alpha^s$ is $m(a,s)$-torsion for all $s \ge 1$.
\end{theorem}
\begin{proof}
  If $\tilde{\alpha}: X \to K(\ZZ, 3)$ factors as $X \to BPU_a \to K(\ZZ,3)$, then applying
  $\Omega(\cdot)$ shows that $\alpha : \Omega(X) \to K(\ZZ, 2)$
  factors through $\eta: PU_a \to K(\ZZ,2)$. The result now follows from our previous determination of the order of $\eta$.
\end{proof}

\begin{remark}
    For $s=1$, this theorem indicates that $\alpha$ is $a$-torsion. For $s=a$, it means that $\alpha^a$ is $1$-torsion, i.e.~zero, so
    $\alpha$ must in particular be nilpotent.
\end{remark}

  Let $\beta$ denote the \textit{unreduced Bockstein}, any one of the natural transformations of cohomology groups that appears as the boundary map in the
  long exact sequence
  \begin{equation*}
    \xymatrix{ \ar[r] & \rHoh^i(X, \ZZ) \ar^{\times r}[r] & \rHoh^i(X, \ZZ) \ar[r] & \rHoh^i(X, \ZZ/r) \ar^{\beta}[r] & \rHoh^{i+1}(X, \ZZ)
      \ar[r] & }
  \end{equation*}
  Each unreduced Bockstein is a natural transformation of cohomology functors, and consequently may be represented by a map $\beta: K(\ZZ/r, i)
  \to K(\ZZ, i+1)$

\begin{corollary} \label{c:kxz}
    The Bockstein $\beta:K(\ZZ/r,2)\rightarrow K(\ZZ,3)$ does not factor through any $BPU_a\rightarrow K(\ZZ,3)$.
    \begin{proof}
      The following argument appears for $r=2$ in Atiyah-Segal~\cite[proof of Proposition 2.1.v]{atiyah-segal}.
      The class $\beta$ has additive order $r$. Since $K(\ZZ/r, 2)$ is simply-connected and since $\Hoh_2(K(\ZZ/r, 2), \ZZ)$ is torsion, it
      follows from the universal-coefficients theorem that $\rHoh^i(K(\ZZ/r, 2), \ZZ) =0$ for $i < 3$. The usual arguments~\cite{mosher_tangora_1968} used to show that
      the cohomology suspension is an isomorphism apply in this instance to show that the cohomology suspension
      $\Hoh^3(K(\ZZ/r, 2), \ZZ) \to \Hoh^2(K(\ZZ/r,1), \ZZ)$ is an isomorphism. In particular $\xi= \Omega(\beta)$ is a generator of
      $\Hoh^2(K(\ZZ/r, 1), \ZZ) = (\ZZ/r)\xi$. It is generally known that $\Hoh^*(K(\ZZ/r,1), \ZZ) = \Hoh^*(B\ZZ/r, \ZZ)$ is
      $\ZZ[\xi]/(r\xi)$, see \cite [Chapter 6]{weibel_introduction_1994}, from which it follows that $\ord(\xi^n) = r$ for all $n\ge 1$. In
      particular $\xi$ is not nilpotent in $\Hoh^*(K(\ZZ/r, 1), \ZZ)$.
    \end{proof}
\end{corollary}

Recall that a map $f: X \to Y$ is an $n$-equivalence if $\pi_k(f):\pi_k(X)\rightarrow\pi_k(Y)$ is an isomorphism for $k<n$ and is surjective for
$k=n$.

\begin{lemma}
    Let $f:X\rightarrow Y$ be an $n$-equivalence. Then,
    \begin{equation*}
        f^*:\Hoh^k(Y,\ZZ)\rightarrow\Hoh^k(X,\ZZ)
    \end{equation*}
    is an isomorphism for $k<n$ and is an injection for $k=n$.
    \begin{proof}
        By the Whitehead theorem~\cite[Theorem 10.28]{switzer}, $f$ induces an isomorphism on integral
        homology in degrees less than $n$ and a surjection in degree $n$.
        Now apply the universal-coefficients theorem
        \begin{equation*}
            0\rightarrow\Ext(\Hoh_{k-1}(Y),\ZZ)\rightarrow\Hoh^k(Y,\ZZ)\rightarrow\Hom(\Hoh_k(Y,\ZZ),\ZZ)\rightarrow0.
        \end{equation*}
    \end{proof}
\end{lemma}

\begin{corollary}
    Let $X$ be a CW-complex, and let $\sk_{n+1}X$ denote the $n+1$-skeleton of $X$, so that
    $i:\sk_{n+1} X\rightarrow X$ is an $n$-equivalence by cellular approximation.
    Then
    \begin{equation*}
        \Omega \sk_{n+1} X\rightarrow \Omega X
    \end{equation*}
    is an $n-1$-equivalence, and so $\Hoh^k(\Omega X,\ZZ)\rightarrow\Hoh^k(\Omega \sk_{n+1} X,\ZZ)$ is an isomorphism for $k<n-1$ an 
    injection for $k=n-1$.
    \begin{proof}
        Taking loops sends $n$-equivalences to $n-1$-equivalences.
    \end{proof}
\end{corollary}

\begin{example}\label{ex:kzl}
  Fix a positive integer $r$. Consider the Eilenberg-MacLane space $K(\ZZ/r, 2)$. We know there is a class $\beta \in \Hoh^3(K(\ZZ/r, 2), \ZZ)$
  such that the cohomology suspension $\Omega(\beta)$ is a generator, $\xi$, for
  $\Hoh^2(K(\ZZ/r, 1), \ZZ)$. We also know that $\ord(\xi^n) = r$ for all $n \ge 1$, as was mentioned in the proof of corollary \ref{c:kxz}.

 If we take instead the CW-complex $f_a:X_{a, r} = \sk_{a+1} K(\ZZ/r,2)\rightarrow K(\ZZ/r,2)$, with
  $a \ge 3$, which may be assumed to be a finite CW-complex, then there are isomorphisms arising from the inclusion of the skeleton
  \begin{align*}
    f_a^*: \Hoh^i(K(\ZZ/r, 2), \ZZ) & \isomto \Hoh^i(X_{a,r} , \ZZ)  \qquad &\text{ for $i<a$} \\
    \Omega(f_a)^*: \Hoh^{i-1}( K(\ZZ/r,1), \ZZ) &\isomto \Hoh^{i-1}(\Omega(X_{a,r}), \ZZ) \qquad &\text{for $i<a$}
  \end{align*}
  Denote the class $f_a^*(\beta)$ by $\gamma_a$.
  Since $a\geq 3$, $\gamma_a$ has order $r$. Denote the cohomology suspension $\Omega(\gamma_a)$ by $\alpha$. The
  cohomology suspension being natural, we have $\alpha = \Omega(\gamma_a) =
  f^*(\Omega (\beta)) = f^*(\xi)$. In particular, $\alpha^j$ has order
  $r$ provided it lies in the range where $f_a^*$ is an inclusion,
  i.e.~provided it lies in $\Hoh^i( \Omega(X_{a,r}, \ZZ)$ with $i \le a-1$,  which is to say $j \le \frac{a-1}{2}$. If therefore $\gamma_a : X_{a,r} \to K(\ZZ, 3)$ is to factor
  $X_{a,r} \to BPU_N \to K(\ZZ,3)$, we must have $r \mid m(N, [\frac{a-1}{2}]) $, from which it follows that $n(r, [\frac{a-1}{2}]) \mid N$ by
  corollary \ref{c:useofn}. In particular
  \begin{equation*}
    \per(\gamma_a) = r \qquad \text{ but} \qquad n\left(r, \left[\frac{a-1}{2}\right]\right) | \ind(\gamma_a) 
  \end{equation*}

Letting $a \to \infty$, we obtain in this way for a given $r$, a sequence of spaces $X_{a,r}$ having in each case a class
  $\gamma_a \in \Hoh^3(X_{a,r}, \ZZ)$ such that $\per(\gamma_a) = r$ but $\ind(\gamma_a) \to \infty$ as $a \to \infty$.
\end{example}

\begin{example}\label{ex:kz2}
    Using the example and theorem~\ref{thm:bound}, one finds that for the class $\gamma_5$ just considered on $\sk_6 K(\ZZ/2,2)$,
    \begin{equation*}
        2^2|\ind(\gamma_3)|2^6.
    \end{equation*}
    In this range, $\pi_i^s$ for $i=1,\ldots,5$, the stable homotopy of $B\ZZ/2$ is just the $2$-primary part of the stable homotopy of spheres, except for $\pi_4^s B\ZZ/2=\ZZ/2$.
\end{example}

\subsection{Remarks on the descent spectral sequence for $K(\ZZ/r,2)$}

It was shown in corollary~\ref{c:kxz} that $\ind(\beta)=\infty$ for the Bockstein $\beta\in\Br'(K(\ZZ/r,2))$. Now, suppose
that $\ind_{\K}(\beta)$ were finite. Then, $\ind_{\K}(\gamma)$ would be bounded by $\ind_{\K}(\beta)$ for
every $r$-torsion class $\gamma$ on every space $X$, whereas above there are examples of finite CW-complexes
and classes $\gamma$ with $\per(\gamma)=r$ and $\ind(\gamma)=\ind_{\K}(\gamma)$ arbitrarily large. It follows that
$\ind_{\K}(\beta)$ is infinite.

Let $c_k$ be the non-negative generator of $\Eoh_{k}^{0,0}\subseteq\Hoh^0(K(\ZZ/r,2),\ZZ)$ in the descent spectral
sequence for $\KU(K(\ZZ/r,2))_{\beta}$. The elements $d_k(c_k)$ are obstructions to $\ind(\beta)=c_k$. Note that $c_2=1$
and $d_2^{\beta}(1)=\beta$.

\begin{proposition}
    The obstructions to the finiteness of $\ind_{\K}(\beta)$ in the descent spectral sequence on $K(\ZZ/r,2)$
    are all of finite order in $\Eoh_k^{k,1-k}$, and infinitely many of them are non-zero.
    \begin{proof}
      In fact, the obstructions all lie in subquotients of $\Hoh^p( K(\ZZ/r, 2), \ZZ(q/2))$
      with $p>0$, which are finitely generated groups, and were observed in section \ref{s:primediv} to be $r$-primary,
      hence torsion, and consequently finite.
      
      If the obstructions were to vanish for sufficiently large $k$, then $\ind_{\K}(\beta)$ would be finite,
      a contradiction.
    \end{proof}
\end{proposition}

Fix a prime $\ell$. The first potentially non-zero obstruction in the twisted spectral sequence for $K(\ZZ/\ell,2)$ is $d_3^{\beta}(1)$.
This is in fact non-zero. Namely, $d_3^{\beta}(1)=\beta\in\Hoh^3(K(\ZZ/\ell,2),\ZZ)$.
It is important to know for computations the next non-zero obstruction in the twisted spectral sequence of $K(\ZZ/\ell,2)$.
The following proposition gives a partial answer for all primes $\ell$, and, together with remark~\ref{ex:kz2}, this shows
that $d_5^{\beta}(2)$ is non-zero for $K(\ZZ/2,2)$.

\begin{proposition}
    The next non-zero obstruction in the twisted spectral sequence for $K(\ZZ/\ell,2)$ is one of
    $d_5^{\beta}(\ell\cdot 1),\ldots,d_{2\ell+1}^{\beta}(\ell\cdot 1)$.
    \begin{proof}
        One knows from example~\ref{ex:kzl} that the index of the class $\gamma_{2\ell+1}$ on $$\sk_{2\ell+2}K(\ZZ/\ell,2)$$ is at least $\ell^2$.
        If the differentials $d_5^{\beta}(\ell\cdot 1),\ldots,d_{2\ell+1}^{\beta}(\ell\cdot 1)$ were all zero, then, by lemma~\ref{lem:indkind}, the descent spectral sequence would give $\ind(\gamma_{2\ell+1})=\ell$,
        a contradiction.
    \end{proof}
\end{proposition}

\begin{corollary}
    In the twisted spectral sequence for $K(\ZZ/2,2)$, the obstruction $d_5^{\beta}(2)$ is non-zero.
    \begin{proof}
        This is a special case of the proposition.
    \end{proof}
\end{corollary}

\begin{corollary}
    If $X$ is a finite CW-complex of dimension at most $6$, and if $\alpha\in\Br(X_{\topo})$ has $\per(\alpha)=2$, then $\ind(\alpha)|8$.
    \begin{proof}
        It follows from a straightforward computation in the Serre spectral sequence for $K(\ZZ/2,1)\rightarrow *\rightarrow K(\ZZ/2,2)$
        that $\Hoh^5(K(\ZZ/2,2),\ZZ)=\ZZ/4$.
    \end{proof}
\end{corollary}

In this case our knowledge of the cohomology of $K(\ZZ/2,2)$ gives us better bounds than are obtained via theorem~\ref{thm:bound}. In general, however, it is not
clear to us whether studying the cohomology of $K(\ZZ/r,2)$ should give tighter bounds than~\ref{thm:bound}. We will return to this question in a future work.

\bibliographystyle{amsplain-pdflatex}
\bibliography{brauer,mypapers}

\end{document}